\documentclass{amsart}
\usepackage{graphicx}
\usepackage{amsfonts}
\newtheorem{tm}{Theorem}

\newtheorem{rem}{Remark}
\newtheorem{rems}{Remarks}
\newtheorem{lm}{Lemma}

\newtheorem{nota}{Notation}

\begin{document}

\title{A domain containing all zeros of the partial theta function}
\author{Vladimir Petrov Kostov}
\address{Universit\'e C\^ote d’Azur, CNRS, LJAD, France} 
\email{vladimir.kostov@unice.fr}
\begin{abstract}
We consider the partial theta function, i.e. the sum of 
the bivariate series $\theta (q,z):=\sum _{j=0}^{\infty}q^{j(j+1)/2}z^j$ 
for $q\in (0,1)$, $z\in \mathbb{C}$. We show that for any value of the 
parameter $q\in (0,1)$ all zeros of the function $\theta (q,.)$ belong to 
the domain 
$\{ {\rm Re}~z<0, |{\rm Im}~z|\leq 132\}$$\cup$$\{ {\rm Re}~z\geq 0, 
|z|\leq 18\}$.


\end{abstract}
\maketitle 

\section{Introduction}

We consider the bivariate series $\theta (q,z):=\sum _{j=0}^{\infty}q^{j(j+1)/2}z^j$ 
for $q\in (0,1)$, $z\in \mathbb{C}$. We regard $q$ as a parameter and $z$ as a 
variable. This series is convergent and defines an entire function called 
{\em partial theta function}. The terminology is explained by the resemblance 
of this formula with the one for the function 

$$\Theta ^*(q,z):=\sum _{j=-\infty}^{\infty}q^{j(j+1)/2}z^j~,$$
because the latter is connected with the {\em Jacobi theta function}   

$$\Theta (q,z):=\sum _{j=-\infty}^{\infty}q^{j^2}z^j$$
by the formula $\Theta ^*(q,z)=\Theta (q^{1/2},q^{1/2}z)$. The 
word ``partial'' reminds that in the formula for $\theta$ the summation is 
performed from $0$ to $\infty$, not from $-\infty$ to $\infty$. 

Studying the function $\theta$ is motivated by its applications in several 
domains the most recent of which concerns section-hyperbolic polynomials, i.e. 
real univariate polynomials of degree $\geq 2$ with all roots real 
and such that when 
their highest-degree monomial is deleted this gives again a polynomial having 
only real roots. The relationship between $\theta$ and such polynomials 
is explained in \cite{KoSh}. Previous research on section-hyperbolic 
polynomials was performed in \cite{KaLoVi} and \cite{Ost} which in turn was 
based on classical results of Hardy, Petrovitch and Hutchinson 
(see \cite{Ha}, \cite{Pe} and \cite{Hu}). Other domains in which 
the partial theta function is used are statistical physics 
and combinatorics (see \cite{So}), asymptotic analysis (see \cite{BeKi}),  
Ramanujan-type $q$-series 
(see \cite{Wa}) and the theory 
of (mock) modular forms (see \cite{BrFoRh}); see also~\cite{AnBe}. 

In the present paper we prove the following theorem:

\begin{tm}\label{TM}
(1) For {\rm Re}\,$z\geq 0$ the function $\theta (q,.)$ 
has no zeros outside the closed half-disk 
$\{ ${\rm Re}\,$z\geq 0, |z|\leq 18\}$, 
for any $q\in (0,1)$.

(2) For {\rm Re}\,$z<0$ and for any $q\in (0,1)$ the function $\theta (q,.)$ 
has no zeros outside the half-strip 
$\{ ${\rm Re}\,$z<0, |${\rm Im}\,$z|\leq 132\}$.  
\end{tm}

In order to explain the importance of this theorem we recall in 
Section~\ref{secproperties} certain facts about the zeros of $\theta$. Then 
we give an example of a value of $q\in (0,1)$ for which $\theta (q,.)$ 
has a complex conjugate pair of zeros in the right half-plane. 
The proof of the theorem is given in Section~\ref{secprooftm}.

\section{Properties of the function $\theta$
\protect\label{secproperties}}

In the present section we recall some results concerning the function 
$\theta$.  We denote by $\Gamma$ the 
{\em spectrum} of $\theta$, i.e. the set of values of $q$ for which 
$\theta (q,.)$ has a multiple zero (the notion has been introduced by 
B.~Z.~Shapiro in \cite{KoSh}). The following results are proved in~\cite{Ko1}: 

\begin{tm}
(1) The spectrum $\Gamma$ consists of countably-many values of $q$ denoted by 
$0<\tilde{q}_1<\tilde{q}_2<\cdots <\tilde{q}_N<\cdots <1$ with 
$\lim _{j\rightarrow \infty}\tilde{q}_j=1^-$. 

(2) For $\tilde{q}_N\in \Gamma$ the function $\theta (\tilde{q}_N,.)$ 
has exactly one multiple real zero $y_N$ which is negative, of multiplicity
$2$ and is the rightmost of its real zeros.

(3) For $q\in (\tilde{q}_N,\tilde{q}_{N+1}]$ (we set $\tilde{q}_0:=0$) 
the function $\theta$ has exactly 
$N$ complex conjugate pairs of zeros
(counted with multiplicity). All its other zeros are real negative.
\end{tm}

\begin{rems}\label{remsasympt}
{\rm (1) It is proved in \cite{KoSh} that $\tilde{q}_1=0.3092\ldots$. Up to 
$6$ decimals the first $12$ spectral numbers equal (see~\cite{KoSh})

$$\begin{array}{llllll}
0.309249,&0.516959,&0.630628,&0.701265,&0.749269,&0.783984,\\ \\ 
0.810251,&0.830816,&0.847353,&0.860942,&0.872305,&0.881949~.\end{array}$$

(2) It is shown in \cite{Ko1} that for $q\in (0,\tilde{q}_1)$ all zeros of 
$\theta$ are real, negative and distinct. For all $q\in (0,1)$ it is true that 
as $q$ increases, the values of the local minima of $\theta$ between two 
negative zeros increase and the values of its maxima between two negative 
zeros decrease. It is always the rightmost two negative zeros with a minimum 
of $\theta$ between them that coalesce to form a double zero of $\theta$ for 
$q=\tilde{q}_N$ and then a complex conjugate pair for $q={\tilde{q}_N}^+$. For 
any $q\in (0,1)$ the function $\theta (q,.)$ has infinitely-many negative zeros 
and no positive ones; $\theta (q,.)$ is increasing for $x>0$ and tends to 
$\infty$ as $x\rightarrow \infty$; there is no finite accumulation point for 
the zeros of~$\theta (q,.)$.

(3) In \cite{Ko2} the following asymptotic expansions of $\tilde{q}_N$ and 
$y_N$ are given:} 

\begin{equation}\label{eqasympt}
\begin{array}{ccll}\tilde{q}_N&=&1-(\pi /2N)+(\log N)/8N^2+O(1/N^2)&,\\ \\  
y_N&=&-e^{\pi}e^{-(\log N)/4N+O(1/N)}&.\end{array}\end{equation} 
\end{rems}

The importance of Theorem~\ref{TM} lies in the fact that while the real zeros 
of $\theta$ remain all negative for any $q\in (0,1)$, no information was known 
about its complex conjugate pairs. It would be interesting to know whether all 
complex conjugate pairs remain (for all $q\in (\tilde{q}_1,1)$) within 
some compact domain in $\mathbb{C}$ (independent of~$q$). 

\begin{lm}\label{lm073}
The function $\theta (0.73,.)$ has exactly one complex conjugate pair of zeros 
inside the open half-disk $\tilde{D}:=\{ |z|<3$, {\rm Re}\,$z>0\}$.
\end{lm}

\begin{proof}
Consider the truncation of $\theta (0.73,.)$ of degree $20$ w.r.t. $x$, 
i.e. the 
polynomial $\theta _{20}:=\sum _{j=0}^{20}0.73^{j(j+1)/2}x^j$. One checks numerically 
(say, using MAPLE) that $\theta _{20}$ has zeros 
$0.03356612894\ldots \pm 2.885381139\ldots i$. These are the only zeros of 
$\theta _{20}$ in the closure of $\tilde{D}$. 
Numerical check shows that the modulus of the restriction of $\theta _{20}$ 
to the border of $\tilde{D}$ is everywhere larger than $0.016$. On the other 
hand the  
sum $\sum _{j=21}^{\infty}|0.73^{j(j+1)/2}x^j|$ is 
$\leq \sum _{j=21}^{\infty}0.73^{j(j+1)/2}3^j<3\times 10^{-22}$. By the Rouch\'e 
theorem the functions $\theta (0.73,.)$ and $\theta _{20}$ have 
one and the same number of zeros inside the half-disk $\tilde{D}$.  
\end{proof}

\section{Proof of Theorem~\protect\ref{TM}
\protect\label{secprooftm}}

As for $q\in (0,\tilde{q}_1]$ all zeros of $\theta (q,.)$ are negative (see 
Remarks~\ref{remsasympt}), we prove Theorem~\ref{TM} only for 
$q\in (\tilde{q}_1,1)$.

\subsection{The Jacobi theta function}

In the proof of Theorem~\ref{TM} we use the Jacobi theta function 
$\Theta (q,z):=\sum _{j=-\infty}^{\infty}q^{j^2}z^j$. By the 
{\em Jacobi triple product} 
one has 

$$\Theta (q,z^2)=\prod _{m=1}^{\infty}(1-q^{2m})(1+z^2q^{2m-1})
(1+z^{-2}q^{2m-1})$$
from which for the function 
$\Theta ^*(q,z):=\Theta (q^{1/2},q^{1/2}z)=
\sum _{j=-\infty}^{\infty}q^{j(j+1)/2}z^j$ one deduces the formula

\begin{equation}\label{eqTheta}
\Theta ^*(q,z)=\prod _{m=1}^{\infty}(1-q^m)(1+zq^m)(1+q^{m-1}/z)~.
\end{equation}

\begin{nota}\label{NOTA}
{\rm We set 

$$\begin{array}{lclclc}
s_m:=1+q^{m-1}/z&,&t_m:=1+zq^m&,&Q:=\prod _{m=1}^{\infty}(1-q^m)&,\\ \\  
P:=\prod _{m=1}^{\infty}t_m&&{\rm and}&& 
R:=\prod _{m=1}^{\infty}s_m&.\end{array}$$ 
Thus $\Theta ^*=QPR$.}
\end{nota}

\subsection{Proof of part (1)}

We begin with the observation that 
for any factor $s_m$ (see (\ref{eqTheta}) and Notation~\ref{NOTA}) 
one has 

$$s_m=1+\bar{z}q^{m-1}/|z|^2~~~\, \, {\rm hence}~~~\, \, 
|s_m|\geq {\rm Re}\, s_m\geq 1~~~\, \, {\rm for}~~~\, \, {\rm Re}\, z\geq 0~.$$
Clearly, for any factor $t_m$ it is true that $|t_m|\geq$~Re\,$t_m\geq 1$ 
and $|t_m|\geq |zq^m|$ for Re~$z\geq 0$.

Further in the proof of 
Theorem~\ref{TM} we subdivide the interval $(0,1)$ to which $q$ belongs into 
intervals of the form

\begin{equation}\label{eqqn}
\begin{array}{lccccc}
q\in (1-1/(n-1)~,~1-1/n]&,&n\in \mathbb{N}&,&n\geq 3&{\rm and}\\ \\ 
q\in (\tilde{q}_1~,~1/2]&.&&&&\end{array}
\end{equation} 

\begin{nota}\label{notathetaG}
{\rm We set $\theta :=\Theta ^*-G$, where $G:=\sum _{j=-\infty}^{-1}q^{j(j+1)/2}z^j$, 
and $u:=2e^{(\pi ^2/6)}=10.36133664\ldots$.}
\end{nota}

\begin{rem}\label{remthetaG} 
{\rm Clearly, for $|z|>1$ one has $|G|\leq \sum _{j=1}^{\infty}1/|z|^j=1/(|z|-1)$. 
In particular, for $|z|\geq 18$ (resp. for $|z|\geq u$) one has 
$|G|\leq 1/17$ (resp. $|G|\leq 1/(u-1)$).}
\end{rem} 

Suppose first that $q\in (1/2,1)$.
We show that for $|z|\geq u$, Re\, $z\geq 0$ one has 
$|\Theta ^*|>|G|$ from which part (1) of the theorem  
follows. 

\begin{lm}\label{lmQ}
For $q\leq 1-1/n$, $n\in \mathbb{N}$, $n\geq 2$, one has 
$Q\geq e^{(\pi ^2/6)(1-n)}$.
\end{lm}

The lemma is a particular case of Lemma~4 in \cite{Ko3}. 

Consider the product 
$P_0:=\prod _{m=1}^nt_m$. It follows from 
$|t_m|\geq |zq^m|$ that $|P_0|\geq |z|^nq^{n(n+1)/2}$. The first line of 
conditions (\ref{eqqn}) implies 

\begin{equation}\label{eqP0}
\begin{array}{cclcl}
|P_0|&\geq&|z|^n(1-1/(n-1))^{n(n+1)/2}&=&|z|^n(1-1/(n-1))^{(n-1)(n+2)/2+1}\\ \\ 
&\geq&|z|^n4^{-(n+2)/2}(1-1/(n-1))&\geq&|z|^n4^{-(n+3)/2}~~~=~~~|z|^n2^{-(n+3)}~;
\end{array}
\end{equation} 
we use the 
inequalities 
\begin{equation}\label{eq4}
(1-1/(n-1))^{n-1}\geq 1/4
\end{equation} 
and $1-1/(n-1)\geq 1/2$ which hold true for $n\geq 3$.  

Set $P_1:=\prod _{m=n+1}^{\infty}t_m$.  
Hence we have $|P_1|\geq 1$, $|R|\geq 1$ and 

$$\begin{array}{cclcl}
|\Theta ^*|&=&Q|P_0||P_1||R|&\geq&e^{(\pi ^2/6)(1-n)}|z|^n2^{-(n+3)}\\ \\ 
&=&(e^{(\pi ^2/6)}/2^3)(|z|/2e^{(\pi ^2/6)})^n&.&\end{array}$$
Obviously, for $|z|\geq u$ 
one has $|z|/2e^{(\pi ^2/6)}\geq 1$. As $e^{(\pi ^2/6)}/2^3=0.64\ldots >1(u-1)$, one 
obtains the inequalities $|\Theta ^*|>1/(u-1)\geq |G|$ 
which proves part~(1) of the theorem 
for $q\in (1/2,1)$ (because $u<18$). 

Suppose that $q\in (\tilde{q}_1,1/2]$. 
In this case for $|z|\geq 18$ and 
Re\,$z\geq 0$ 
one has $|t_1|\geq 18\tilde{q}_1$, $|t_m|\geq 1$, $|s_m|\geq 1$ for 
$m\in \mathbb{N}$ and (by Lemma~\ref{lmQ} with $n=2$)  
$Q\geq e^{-\pi ^2/6}$, so 
$|\Theta ^*|\geq e^{-\pi ^2/6}18\tilde{q}_1>1>1/(|z|-1)\geq |G|$.

\subsection{Proof of part (2)}

The proof of part (2) is also based on formula (\ref{eqTheta}). We aim to show 
that for Re$z<0$ and $|$Im\,$z|\geq 132$ one has $|\Theta ^*|>|G|$. 
The following 
technical result is necessary for the estimations and for the understanding of 
Figure~\ref{figab}:

\begin{lm}\label{lmtechnical}
For $x\in [0,0.683]$ one has $\ln (1-x)\geq -x-x^2$ 
with equality only for $x=0$.
\end{lm}

\begin{proof}
We set 
$\zeta (x):=\ln (1-x)+x+x^2$, so $\zeta (0)=0$. 
As $\zeta '=-1/(1-x)+1+2x=x(1-2x)/(1-x)$ which is 
nonnegative on $[0,1/2]$ and positive on $(0,1/2)$, one has  
$\zeta (x)>0$ for $x\in (0,1/2]$. On $(1/2,1)$ one has $\zeta '<0$, so $\zeta$ 
is decreasing. As $\lim _{x\rightarrow 1^-}\zeta =-\infty$, $\zeta$ has a single 
zero on $[1/2,1)$. Numerical computation shows that this zero is 
$>0.683$ which proves the lemma. 
\end{proof}

To estimate the factor $R$ we use the following lemma:

\begin{lm}\label{lmestimR}
For $q\in (1-1/(n-1),1-1/n]$, $n\geq 2$, and 
$|${\rm Im\, }$z|\geq b\geq 1.5$ one has 
$|R|\geq e^{-n(b+1)/b^2}$.
\end{lm}

\begin{proof}
Indeed, the condition $b\geq 1.5$ implies $1/|z|<0.683$, so one can apply 
Lemma~\ref{lmtechnical}: 

$$\begin{array}{ccccl}
\ln |R|&=&
\sum _{m=1}^{\infty}\ln |s_m|&\geq&
\sum _{m=1}^{\infty}\ln (1-q^{m-1}/|z|)\\ \\ &\geq& 
-\sum _{m=1}^{\infty}(q^{m-1}/|z|+q^{2m-2}/|z|^2)&=&
-1/(1-q)|z|-1/(1-q^2)|z|^2\\ \\ 
&>&
-(|z|+1)/(1-q)|z|^2&&\end{array}$$
which for $q\leq 1-1/n$ is $\geq -n(|z|+1)/|z|^2\geq -n(b+1)/b^2$. 
\end{proof}

\begin{figure}[htbp]
\centerline{\hbox{\includegraphics[scale=0.7]{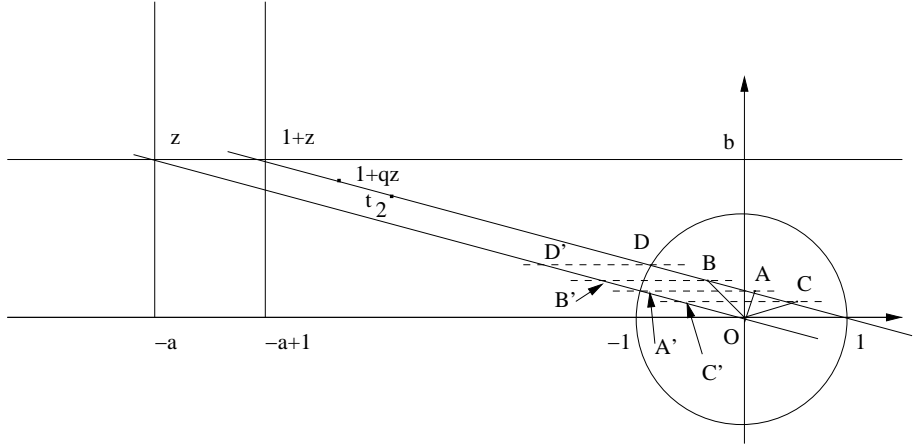}}}
    \caption{The points $1+q^mz$.}
\label{figab}
\end{figure}

We identify the complex numbers and the points in $\mathbb{R}^2$ representing 
them. On Fig.~\ref{figab} we represent the points 

$$\begin{array}{lcll}z=-a+bi~~~\, (a>0)&,& 
t_0:=1+z=1-a+bi&,\\ \\   
t_1:=1+qz=1-qa+qbi&~~~{\rm and}~~~&t_2:=1+q^2z=1-q^2a+q^2bi&.\end{array}$$ 
The last three of them are 
situated on the straight line $\mathcal{L}$ 
passing through $1+z$ and $(1,0)$. 

In what follows we assume that $b\geq 0$. The set of zeros of $\theta$ being 
symmetric w.r.t. the real axis this leads to no loss of generality.

The point $A\in \mathcal{L}$ is such that 
the segment $OA$ is orthogonal to $\mathcal{L}$. An easy computation shows that 
$A=(b^2/(a^2+b^2),ab/(a^2+b^2))$ (one has to use the fact that the vector 
$(-a,b)$ is collinear with $\mathcal{L}$). The points $B$ and $C$ belong to 
$\mathcal{L}$. The unit circumference intersects the line $\mathcal{L}$ 
at $(1,0)$ and $D$. Denote by $\tilde{\Delta}$ the length 
$\| [A,(1,0)]\|$ of the segment 
$[A,(1,0)]$. The points $B$ and $C$ are defined such that  

\begin{equation}\label{eqsegments}
\| [A,B]\| =\| [A,C]\| =0.317\, \tilde{\Delta}~~~\, {\rm and}\, ~~~
\| [C,(1,0)]\| =\| [D,B]\| =0.683\, \tilde{\Delta}~.
\end{equation}  
The line $\mathcal{L}'$ is parallel to $\mathcal{L}$. It passes through 
the points $z$ and $O:=(0,0)$. The points $B'$, $A'$, $C'$ and $D'$ belong to 
$\mathcal{L}'$. The lines $BB'$, $AA'$, $CC'$ and $DD'$ are parallel 
to the $x$-axis. Hence the segments $[D,B]$, $[C,(1,0)]$, $[D',B']$ 
and $[C',O]$ 
are of length $0.683\, \tilde{\Delta}$ while $[B,A]$, 
$[A,C]$, $[B',A']$ and 
$[A',C']$ are of length $0.317\, \tilde{\Delta}$. 

Our aim is to estimate the product 
$|P|:=\prod _{m=1}^{\infty}|t_m|=\prod _{m=1}^{\infty}|1+q^mz|$. 

\begin{nota}\label{NOTA1}
{\rm (1) We set}

\begin{equation}\label{eqPPPP}
|P|:=\tilde{P}P^{\ddagger}P^{\sharp}P^{\dagger}~,
\end{equation}
{\rm where $\tilde{P}$, $P^{\ddagger}$, $P^{\sharp}$ and $P^{\dagger}$ 
are the products of the moduli 
$|t_m|$ for which the point 
$t_m$ belongs to the segment $[1+z,D]$, $[D,B]$, $[B,C]$ and 
$[C,(1,0)]$ respectively.

(2) We denote by $t_{m_0}$, $t_{m_0+1}$, $\ldots$ the points $t_m$ 
belonging to 
the segment $[C,(1,0)]$ and we set $c_{m_0}:=C$, $c_{m_0+k}:=1+(C-1)q^k$. Hence 
$|t_{m_0+k}|\geq |c_{m_0+k}|$ with equality only if $t_{m_0}=C$.}
\end{nota}

\begin{lm}\label{lmPddagger}
For $q\in (1-1/(n-1),1-1/n]$, $n\geq 2$, one has 
$P^{\dagger}\geq e^{-1.149489n}$ and $P^{\ddagger}\geq e^{-1.149489n}$ (where 
$1.149489=0.683+0.683^2$).
\end{lm}
 
\begin{proof}
We notice first that the segment $[C,(1,0)]$ is 
of length $<0.683$. Hence $|t_{m_0+k}|\geq |c_{m_0+k}|\geq (1-0.683q^k)$, $k=0$, 
$1$, $\ldots$, so we can use Lemma~\ref{lmtechnical} to get

$$\begin{array}{cclcl}
\ln P^{\dagger}&\geq&\ln (|c_{m_0}||c_{m_0+1}|\cdots )&\geq&
\ln(\prod _{m=1}^{\infty}(1-0.683q^{m-1}))\\ \\ 
&\geq&-\sum _{m=1}^{\infty}(0.683q^{m-1}+0.683^2q^{2m-2})&&\\ \\ 
&=&-0.683/(1-q)-0.683^2/(1-q^2)
&>&-1.149489/(1-q)~.\end{array}$$  
Thus $P^{\dagger}\geq e^{-1.149489n}$. 
Next, the distance between any two consecutive points $1+q^mz$ and 
$1+q^{m+1}z$ belonging to $[B,D]$ is greater than the distance between any two 
such points belonging to $[C,(1,0)]$. Denote by $U_1$, $U_2$, $\ldots$, $U_r$ 
the points $t_m$ belonging to the segment $[B,D]$, where $U_1$ 
(resp. $U_r$) is closest to 
$B$ (resp. to $D$). Then $|U_1|\geq |c_{m_0}|$, $|U_2|\geq |c_{m_0+1}|$, 
$\ldots$, $|U_r|\geq |c_{m_0+r-1}|$. As $|c_m|\leq |t_m|<1$ for 
$m\geq m_0$, one has 
$P^{\ddagger}\geq |c_{m_0}||c_{m_0+1}|\cdots |c_{m_0+r-1}|>
|c_{m_0}||c_{m_0+1}|\cdots \geq e^{-1.149489n}$.
\end{proof}

\begin{lm}\label{lmPsharp}
There are $\leq \mu _1$ factors $|t_m|$ 
in $P^{\sharp}$, where 
$$\begin{array}{cclc}
\mu _1&:=&\ln \lambda _1/\ln (1/q)+1&{\rm and}\\ \\ 
\lambda _1&:=&(0.634+0.683)/0.683=1.928\ldots&\end{array}$$ 
with $\ln \lambda _1=0.6566\ldots$. Hence   
$$P^{\sharp}\geq (b^2/(a^2+b^2))^{\mu _1/2}=(1/(\beta ^2+1))^{\mu _1/2}~~~,~~~ 
\beta :=a/b~.$$
\end{lm}

\begin{proof}
Consider the points $C'$, $A'$ and $B'$ and the numbers $zq^{m_1}\in [B',A']$ 
and $zq^{m_2}\in [A',C']$ closest to $B'$ and $C'$ respectively. The lengths 
of 
the segments $[B',A']$, $[A',C']$ and $[C',O]$ (see (\ref{eqsegments})) 
imply $|zq^{m_1}|/|zq^{m_2}|\leq \lambda _1$, i.e. 
$m_2-m_1\leq \lambda _1/\ln (1/q)$. The number of factors $|t_m|$ 
in $P^{\sharp}$ equals $m_2-m_1+1$ from which one deduces the first 
claim of the lemma. 
All factors $|t_m|$ in $P^{\sharp}$ are $<1$ and  
$\geq \| [O,A]\| =b/(a^2+b^2)^{1/2}$ from which the second claim of the lemma 
follows. 
\end{proof}

\begin{rem}\label{rrem}
{\rm When $q\in (1/2,1)$ and (\ref{eqqn}) holds true, 
then $n/(n-1)\leq 1/q<(n-1)/(n-2)$. As 
for $x\in (0,1)$ one has $x-x^2/2<\ln (1+x)<x$ (by the Leibniz criterium for 
alternating series), one obtains the inequalities}  

\begin{equation}\label{mu10}
\begin{array}{ll}
(2n-3)/2(n-1)^2~=~1/(n-1)-1/2(n-1)^2~<~\ln (1/q)<1/(n-2)&{\rm hence}\\ \\ 
(\ln \lambda _1)(n-2)+1~<~\mu _1~\leq ~
\mu _1^0~:=~(\ln \lambda _1)(2(n-1)^2/(2n-3))+1&.\end{array}
\end{equation}
\end{rem}

\begin{lm}\label{lmL}
For $q\in (\tilde{q}_1,1)$ and $b\geq \max (a,132)$ one has 
$|\Theta ^*|>|G|$.

\end{lm}

\begin{proof}
Prove first the lemma for $q\in (1/2,1)$, 
see the first line of (\ref{eqqn}). Set again $P_0:=\prod_{m=1}^nt_m$. Hence 

$$|t_m|\geq \, {\rm Im}\, t_m~~~\,  
{\rm and}~~~\, |P_0|\geq 
\prod _{m=1}^nbq^m=b^nq^{n(n+1)/2}\geq b^n2^{-(n+3)}~,$$ 
see (\ref{eqP0}) and (\ref{eq4}). 
With $Q$, $P^{\dagger}$, $P^{\ddagger}$ 
and 
$P^{\sharp}$ defined in Notations~\ref{NOTA} and \ref{NOTA1} one has 

\begin{equation}\label{eqQQQ}
|\Theta ^*|\geq Q|P_0|P^{\ddagger} 
P^{\sharp}P^{\dagger}R~.
\end{equation}
Indeed, if $b\geq 132$ and if $q$ satisfies the first line of 
conditions (\ref{eqqn}), then 

\begin{equation}\label{qm}
q^m\geq q^n\geq (1-1/(n-1))^{n-1}(1-1/(n-1))\geq (1/4)(1-1/(n-1))\geq 1/8
\end{equation} 
and $|bq^m|\geq |bq^n|\geq 132/8>1$. This means that all factors 
$|t_m|$ in $|P_0|$ are 
$>1$. Moreover,   
some factors $|t_m|$ with $|t_m|>1$ which are present in $|\Theta ^*|$ 
(i.e. in $\tilde{P}$, see (\ref{eqPPPP})) might 
be missing in the right-hand side of (\ref{eqQQQ}). 
Recall that each of the factors $P^{\dagger}$ and 
$P^{\ddagger}$ is minorized by $e^{-1.149489n}$ and that $|R|\geq e^{-(b+1)n/b^2}$, 
see Lemmas~\ref{lmPddagger} and 
\ref{lmestimR}. 
Recall also that by Lemma~\ref{lmPsharp}, 

$$
P^{\sharp}\geq (1/(\beta ^2+1))^{\mu _1/2}\geq 2^{-\mu _1/2}~~~
{\rm (because~~~}\beta =a/b\leq 1~{\rm )}$$
and that $\mu _1\leq \mu _1^0$, see (\ref{mu10}). 
Hence the right-hand side of (\ref{eqQQQ}) is 
$$\begin{array}{cccl}
\geq&H&:=&e^{(\pi ^2/6)(1-n)}b^n2^{-(n+3)}e^{2(-1.149489n)}\\ \\ 
&&&\times e^{-(\ln 2)((\ln \lambda _1)2(n-1)^2/(2n-3)+1)/2}e^{-(b+1)n/b^2}~.\end{array}$$
Taking into account that 

\begin{equation}\label{eq2n}
2(n-1)^2/(2n-3)=n-1/2+1/2(2n-3)~,
\end{equation} 
we represent the expression $H$ in the form $e^{K_1n+K_0}$, where 

$$\begin{array}{ccl}K_1&:=&-\pi ^2/6+
\ln b-\ln 2-2.298978-(\ln 2)(\ln \lambda _1)/2-(b+1)/b^2~~~,\\ \\ 
K_0&:=&\pi ^2/6-3\ln 2+(\ln 2)(\ln \lambda _1)/4-(\ln 2)/2-
(\ln 2)(\ln \lambda _1)/4(2n-3)~.\end{array}$$
Recall that $n\geq 3$, see the first line of (\ref{eqqn}). 
The sum $K_0$ is minimal for $n=3$. 
For $b\geq 132$ one has $K_1>0$ and $K_0|_{n=3}>0>-\ln (b-1)$ 
which implies the inequalities 

$$|\Theta ^*|\geq Q|P_0|P^{\ddagger}P^{\sharp}P^{\dagger}R>
1/(b-1)\geq 1/(|z|-1)\geq |G|~.$$

Prove the lemma for $q\in (\tilde{q}_1,1/2]$. One has $Q\geq e^{-\pi ^2/6}$ 
(Lemma~\ref{lmQ} with $n=2$), $R\geq e^{-2(b+1)/b^2}\geq e^{-2\times 133/132^2}$ 
(Lemma~\ref{lmestimR} 
with $n=2$) and $P^{\dagger}\geq e^{-2.298978}$ 
(Lemma~\ref{lmPddagger} 
with $n=2$). 

\begin{lm}\label{lmtwofactors}
For $q\in (\tilde{q}_1,1/2]$ the product $P^{\ddagger}P^{\sharp}$ 
contains at most two factors.
\end{lm}

\begin{proof} 
Indeed, consider the line $\mathcal{L}'$, see Fig.~\ref{figab}. One has 
$\| [O,C']\| =0.683\, \tilde{\Delta}$ and 
$\| [O,D']\| <4\times 0.683\, \tilde{\Delta}$, see (\ref{eqsegments}) and the 
lines that follow. 
Hence if the point $zq^m$ 
belongs to the segment $[D',C']$, then this is not the case of the point 
$zq^{m-2}$, because for $q\in (\tilde{q}_1,1/2]$ one has $q^{-2}\geq 4$ (but 
one could possibly have $zq^{m-1}\in [D',C']$). 
\end{proof}

All factors $|t_m|$ of the product 
$P^{\ddagger}P^{\sharp}$ belong to $[b/(a^2+b^2)^{1/2},1)$, therefore 
by Lemma~\ref{lmtwofactors}, 
$P^{\ddagger}P^{\sharp}\geq b^2/(a^2+b^2)$ which for $a\leq b$ is~$\geq 1/2$. 

On the other hand, the moduli of the first three factors $|t_m|$ 
in $\tilde{P}$ are not less than respectively 

$$132\, \tilde{q_1}-1>39.814~~~\, ,~~~\, 132\, \tilde{q_1}^2-1>11.619~~~\, 
{\rm and}~~~\, 
132\, \tilde{q_1}^3-1>2.902$$ 
and the moduli of all other factors $|t_m|$ in $\tilde{P}$ 
(if any) are $\geq 1$, so for $b\geq 132$

$$\begin{array}{ccl}
|\Theta ^*|&\geq&e^{-\pi ^2/6}\times (39.814\times 11.619\times 2.902)\times 
(1/2)\times e^{-2.298978}\times e^{-2\times 133/132^2}\\ \\ &>&12.8~~>~~|G|~.
\end{array}$$
\end{proof}

\begin{lm}\label{lmLL}
For $a\geq b\geq 132$ one has $|\Theta ^*|>|G|$.
\end{lm}

\begin{proof}
Suppose first that $q\in (1/2,1)$. We define $n\geq 3$ from conditions 
(\ref{eqqn}). Recall that the number $\mu _1$ was defined in 
Lemma~\ref{lmPsharp} and that inequalities (\ref{mu10}) hold true. For 
$n\geq 3$ equality (\ref{eq2n}) implies  

\begin{equation}\label{eqmun}
\begin{array}{cclcl}
\mu _1&\leq&(\ln \lambda _1)2(n-1)^2/(2n-3)+1&&\\ \\ 
&=&(\ln \lambda _1)(n-1/2+1/2(2n-3))+1&<&(\ln \lambda _1)n+0.782~,
\end{array}
\end{equation}
because $1/2(2n-3)\leq 1/6$ and $(\ln \lambda _1)(-1/2+1/6)+1=0.7811\ldots$. 

Consider a factor $t_m$ from $P_0:=\prod _{m=1}^nt_m$. One has 

\begin{equation}\label{a2b2}
|t_m|^2=(aq^m-1)^2+b^2q^{2m}\geq 0.9(a^2+b^2)q^{2m}~~;
\end{equation}
this follows from 

\begin{equation}\label{a2b2bis}
(a^2+b^2)q^{2m}/10-2aq^m+1=(aq^m-10)^2/10+(b^2q^{2m}-90)/10\geq 0~~;
\end{equation}
the last inequality results from $b\geq 132$ and (\ref{qm}) 
(remember that if 
$q$ satisfies conditions (\ref{eqqn}) with $n\geq 3$, then the 
inequality (\ref{eq4}) holds true), so $b^2q^{2m}\geq (132/8)^2>90$. 
Set $A:=(a^2+b^2)^{(n-\mu _1)/2}$. Hence 

$$\begin{array}{ccccc}
|P_0|&\geq&(a^2+b^2)^{n/2}q^{n(n+1)/2}(0.9)^{n/2}&\geq&
(a^2+b^2)^{n/2}2^{-(n+3)}(0.9)^{n/2}\\
{\rm and}&&&&\\  
P^{\sharp}&\geq&b^{\mu _1}/(a^2+b^2)^{\mu _1/2}&,&\\{\rm so}&&&&\\  
|P_0|P^{\sharp}&\geq&Ab^{(\ln \lambda _1)(n-2)+1}2^{-(n+3)}(0.9)^{n/2}&.&\end{array}$$
(we use inequalities 
(\ref{mu10})). As $(a^2+b^2)^{1/2}\geq 132\sqrt{2}$ and as 
$n-\mu _1>\omega _1n-0.782$, $\omega _1:=1-(\ln \lambda _1)$, 
see (\ref{eqmun}), 
one obtains the minoration $|P_0|P^{\sharp}\geq e^M$, where 

$$\begin{array}{ccl}
M&:=&(\omega _1n-0.782)\ln (132\sqrt{2})+(\ln b)((\ln \lambda _1)(n-2)+1)\\ \\ 
&&-(n+3)\ln 2+(n/2)\ln 0.9~.\end{array}$$
To estimate $P^{\dagger}$ and 
$P^{\ddagger}$ we use Lemma~\ref{lmPddagger}. 
As in the proof of Lemma~\ref{lmL} one can minorize the right-hand 
side of (\ref{eqQQQ}) by 

$$e^{(\pi ^2/6)(1-n)}e^Me^{2(-1.149489n)}e^{-(b+1)n/b^2}~.$$
This expression is of the form $e^{L_1n+L_0}$ with 

$$\begin{array}{ccl}
L_1&=&-\pi ^2/6+\omega _1\ln (132\sqrt{2}) 
+(\ln b)(\ln \lambda _1)\\ \\ &&-\ln 2+(\ln 0.9)/2-2.298978-(b+1)/b^2~,\\ \\  
L_0&=&\pi ^2/6-0.782\ln (132\sqrt{2})+(\ln b)(-2\ln \lambda _1+1)-3\ln 2~.
\end{array}$$
For $a\geq b=132$ one has $|z|\geq 132\sqrt{2}$, also 
$L_1>0.3044>0$ and     
$L_0=-6.0491\ldots$. For $a\geq b=132$, $n\geq 3$ one has 

$$L_1n+L_0\geq -5.136\ldots>-5.224\ldots =
-\ln (132\sqrt{2}-1)~,$$
i.e. $|\Theta ^*|>1/(|z|-1)\geq |G|$. The functions $L_1n+L_0$ and $e^{L_1n+L_0}$ 
when considered as functions in $b$ (for $n\geq 3$ fixed) are increasing 
while the functions $-\ln (b-1)$ and $1/(b-1)$ are decreasing. Therefore 
one has $|\Theta ^*|>1/(|z|-1)\geq |G|$ for $a\geq b\geq 132$, $n\geq 3$ 
from which for $q\in (1/2,1)$ 
the lemma follows. 

Suppose that $q\in (\tilde{q}_1,1/2]$. One deduces from 
Lemma~\ref{lmtwofactors} (as in the proof of Lemma~\ref{lmL}) 
that $P^{\ddagger}P^{\sharp}\geq b^2/(a^2+b^2)$. On the other hand, consider the 
factors $t_1$ and $t_2$. One can apply to them inequality (\ref{a2b2}) 
with $m=1$ and $2$. Hence $|t_1|>1$, $|t_2|>1$, 

$$\tilde{P}\geq |t_1||t_2|>0.9(a^2+b^2)\tilde{q}_1^3~~~\, {\rm and}~~~\, 
\tilde{P}P^{\ddagger}P^{\sharp}>0.9b^2\tilde{q}_1^3>463~.$$
As in the proof of Lemma~\ref{lmL} we show that $R\geq e^{-2\times 133/132^2}$ 
and $P^{\dagger}\geq e^{-2.298978}$, so finally 

$$|\Theta ^*|=\tilde{P}P^{\ddagger}P^{\sharp}QP^{\dagger}R>
463e^{-\pi ^2/6}e^{-2.298978}e^{-2\times 133/132^2}>8.8>|G|~.$$
\end{proof}

Lemmas~\ref{lmL} and \ref{lmLL} together imply that for $q\in (0,1)$ and 
$z=-a+bi$, $b>132$, 
$a>0$, the function $\theta (q,.)$ has no zeros. Theorem~\ref{TM} is proved.

\end{document}